\documentclass[a4paper,twoside,english]{amsart}
\usepackage{amsmath,amssymb,amsthm,amscd} 
\usepackage{amssymb,fge} 
\usepackage{tikz-cd}
\usepackage{mathtools}
\usepackage{comment}
\usepackage{graphicx}
\usepackage[left=3.2cm,right=3.5cm,top=3cm,bottom=3cm]{geometry}
\usepackage{indentfirst}           
\usepackage{underscore}
\usepackage{enumitem}
\usepackage{relsize}
\usepackage{varwidth}
\usepackage{mathtools} 
\usepackage{hyperref}
\hypersetup{
    colorlinks=true,
    linkcolor=blue,
    filecolor=magenta,      
    urlcolor=cyan,
}
\usepackage{amssymb}
\usepackage[all]{xy}
\usepackage{todonotes}

\newtheorem{thm}{Theorem}
\newtheorem{cor}[thm]{Corollary}
\newtheorem{prop}[thm]{Proposition}
\newtheorem{lem}[thm]{Lemma}

\newcommand\scalemath[2]{\scalebox{#1}{\mbox{\ensuremath{\displaystyle #2}}}}

\theoremstyle{definition} 
\newtheorem{defin}[thm]{Definition}

\theoremstyle{remark}
\newtheorem{remark}[thm]{Remark}

\numberwithin{thm}{section}
\numberwithin{equation}{thm}
\newcommand{\be}{\begin{equation}}
\newcommand{\ee}{\end{equation}}

\usepackage[utf8]{inputenc}
\title[prime-powered images and irreducible polynomials in dynamical semigroups]{ Prime-powered images and irreducible polynomials in dynamical semigroups}
\date{August 2025}
\author[]{Aristaa Bhardwaj}
\author[]{Adrian Boyer-Paulet}
\author[]{Wade Hindes}
\author[]{Emma Qiu}
\author[]{Alexander Sun}

\begin{document}
\begin{abstract} Let $G=\langle x^d+c_1,\dots,x^d+c_s\rangle$ be a semigroup generated under composition for some $c_1,\dots,c_s\in\mathbb{Z}$ and some $d\geq2$. Then we prove that, outside of an exceptional one-parameter family, $G$ contains a large and explicit subset of irreducible polynomials if and only if it contains at least one irreducible polynomial. In particular, this conclusion holds when $G$ is generated by at least $s\geq3$ polynomials when $d$ is odd and at least $s\geq5$ polynomials when $d$ is even. To do this, we prove a classification result for prime powered iterates under $f(x)=x^d+c$ when $c\in\mathbb{Z}$ is nonzero. Namely, if $f^n(\alpha)=y^p$ for some $n\geq4$, some $\alpha,y\in\mathbb{Z}$, and some prime $p|d$, then $\alpha$ and $y^p$ are necessarily preperiodic and periodic points for $f$ respectively. Moreover, we note that $n=4$ is the smallest possible iterate for which one may make this conclusion.       
\end{abstract}
\maketitle

\section{Introduction}
Let $K$ be a field, let $f_1,\dots,f_s\in K[x]$ be a collection of polynomials with coefficients in $K$, and let $G=\langle f_1,\dots,f_s\rangle$ be the semigroup generated by $f_1,\dots,f_s$ under composition. For example, $\langle f\rangle$ is simply the set of iterates of $f$. One problem that arises naturally in arithmetic dynamics is to determine a set of conditions that ensure that the semigroup $G$ contains many irreducible polynomials in $K[x]$. To make this problem more precise,  we say that $G$ \emph{contains a positive proportion of irreducible polynomials over $K$} if \vspace{.05cm}
\[\liminf_{B\rightarrow\infty}\frac{\#\{g\in G\,: \deg(g)\leq B\;\text{and $g$ is irreducible over $K$}\}}{\#\{g\in G\,: \deg(g)\leq B\}}>0.
\vspace{.05cm}
\]
Clearly, a necessary condition that $G$ contains a positive proportion of irreducible polynomials is that it contains at least one such polynomial, but is this sufficient? We prove that the answer to this question is yes, outside of a one-parameter family of exceptional semigroups, when $K=\mathbb{Q}$ and $G$ is generated by unicritical polynomials with integral coefficients, all of the same degree. To do this, we build on previous work in \cite{PreperiodicPointsandABC,hindes2025proportion,hindes2023irreducible}. Moreover, in what follows, irreducible means irreducible over $\mathbb{Q}$.
\vspace{.05cm}
\begin{thm}\label{thm:main+irreducible} Let $G=\langle x^d+c_1,\dots, x^d+c_s\rangle$ for some $d \ge 2$ and some $c_1,\dots,c_s\in\mathbb{Z}$. Then one of the following statements must hold: \vspace{.1cm} 
\begin{enumerate}
    \item[\textup{(1)}] $G$ contains a positive proportion of irreducible polynomials if and only if it contains at least one irreducible polynomial.\vspace{.2cm}     
    \item[\textup{(2)}] $d\geq4$ is even and $\{c_1,\dots, c_s\}\subseteq \big\{y^p-y^{pd}\,,\;y^p\,,\;-y^p\,,\;-y^p-y^{pd}\big\}$  
    for some $y\in\mathbb{Z}$ and some prime $p|d$. \vspace{.2cm} 
      \item[\textup{(3)}] $d\geq5$ is odd and $\{c_1,\dots,c_s\}\subseteq \big\{y^p-y^{pd}\;,\;y^p\big\}$ 
    for some $y\in\mathbb{Z}$ and some prime $p|d$. \vspace{.1cm}  
\end{enumerate}
\end{thm}
There is a dynamical reason for the exceptional semigroups in statements (2) and (3). Namely, in these cases the \emph{critical orbit} (i.e., the semigroup orbit of zero) contains a $p$th powered fixed point for, perhaps the only, irreducible map in the generating set of $G$; see Section 3 for details. On the other hand, outside of the exceptional semigroups above, we prove a more explicit statement: there are $f_1,f_2\in G$ such that one of the following subsets, 
\[
\big\{f_1^3\circ g\,:\,g\in G\big\}
\;\;\;\text{or}\;\;\;
\big\{f_1^3\circ f_2\circ f_1\circ g\,:\,g\in G\big\}
\;\;\;\text{or}\;\;\;
\big\{f_1^{3}\circ f_2^{3}\circ g\,:\,g\in G\big\},
\]
is a set of irreducible polynomials, and we determine when each type is needed in Propositions \ref{prop:no+special+irreducibles}, \ref{prop:special+irreducible+odd}, and \ref{prop:special+irreducible+even} respectively. Furthermore, it is likely that statement (1) of Theorem \ref{thm:main+irreducible} holds without exception; for one, the exceptional cases were successfully resolved in small degree (i.e., when $d=2$ and $d=3$) using rational points techniques \cite{hindes2025proportion}. However, these techniques become impractical as $d$ grows, and so a new method is needed in large degree.       

Moreover, we deduce the following useful consequence of Theorem \ref{thm:main+irreducible}. Namely, if $G$ contains an irreducible polynomial and has sufficiently many generators, then it contains a positive proportion of such polynomials; compare to \cite[Corollary 1.13]{PreperiodicPointsandABC}. With this in mind, we assume throughout that the coefficients $c_1,\dots,c_s\in\mathbb{Z}$ are distinct, since otherwise, we may simply delete generators from $G$.

\begin{cor}\label{cor:irreducible} Let $G=\langle x^d+c_1,\dots, x^d+c_s\rangle$ for some $d \ge 2$ and some $c_1,\dots,c_s\in\mathbb{Z}$. Moreover, assume that $s\geq3$ if $d\geq5$ is odd and $s\geq5$ if $d\geq4$ is even. Then $G$ contains a positive proportion of irreducible polynomials if and only if it contains at least one irreducible polynomial.   
\end{cor}

As in the case of iterating a single function, our irreducbility results stem from the ability to avoid $p$th powers in orbits; see, for instance, \cite{hamblen2015density,jones2008density}. However, this issue is more subtle for semigroups with at least two generators, since the orbits in this case tend to be much larger \cite{bell2023counting}. In particular, the main tool we use to prove the irreduciblity statements above is the following dynamical $p$th power classification theorem; compare to similar results in \cite[Proposition 5.7]{PreperiodicPointsandABC} and \cite[Theorems 2.3 and 2.7]{hindes2025proportion}. \vspace{.1cm}

\begin{thm}\label{thm:pth+powered+iterate+intro} Let $f(x)=x^d+c$ for some nonzero $c\in\mathbb{Z}$ and $d\geq2$. Moreover, assume that $f^n(\alpha)=\epsilon y^p$ for some $\alpha,y\in\mathbb{Z}$, some $\epsilon=\pm{1}$, some prime $p|d$, some $n\geq4$ if $d=2$, and some $n\geq3$ if $d\geq3$. Then $\alpha$ is preperiodic and $\epsilon y^p$ is periodic for $f$ respectively.
\vspace{.1cm}
\end{thm}
\begin{remark} This result fits nicely with the following heuristic in arithmetic dynamics: if an orbit possesses a special arithmetic property (in this case, has a large $p$th power), then there is a good dynamical or geometric reason why (in this case, either $f=x^d$ is itself a $p$th power or the orbit in question is a finite set).      
\end{remark}
\begin{remark} Note that Theorem \ref{thm:pth+powered+iterate} is false when $c=0$: in this case, $f^n(\alpha)$ is a $d$th power for all $\alpha\in\mathbb{Z}$ and all $n\geq1$, even though $|\alpha|\geq2$ is not preperiodic for $x^d$. Moreover, the lower bound on the iterate $n$ in the statement of Theorem \ref{thm:pth+powered+iterate+intro} is the smallest possible lower bound which ensures that $\alpha$ is preperiodic. For example, if $f(x)=x^d-r^d$ for some $r\geq2$, then $f^2(r)=-r^d$ and $\alpha=r$ is not preperiodic for $f$.   
Likewise, if $f(x)=x^2-460$, then we find that $f^3(22)=(114)^2$ and $\alpha=22$ is not preperiodic for $f$. Hence, $n$ cannot be decreased to $3$ when $d=2$ and cannot be decreased to $2$ when $d>2$ in general. In particular, on several fronts, Theorem \ref{thm:pth+powered+iterate+intro} is the strongest possible statement regarding $p$th powered images.        
\end{remark}
An outline of this paper is as follows: we prove the dynamical $p$th power classification theorem in Section \ref{sec:pth+power} and prove the irreducibility results in Section \ref{sec:irreduciblity}. 
\\[5pt] 
\noindent\textbf{Acknowledgements:} We thank the Mathworks honors program at Texas State University for supporting this research.
\section{Prime-powered iterated images}\label{sec:pth+power}
The goal of this section is to prove the following classification result for prime-powered iterated images under unicritical polynomials defined over the integers; compare to Theorem \ref{thm:pth+powered+iterate+intro} from the Introduction. Moreover, see Remark \ref{rem:pth+powers+implies+intro+version} for an explanation of why the result below implies the version from the Introduction. 

To state this result, recall that a point $\alpha$ is called \emph{periodic} if $f^n(\alpha)=\alpha$ for some $n\geq1$ and called \emph{preperiodic} if $f^m(\alpha)$ is periodic for some $m\geq0$; equivalently, $\alpha$ is preperiodic if the orbit of $\alpha$ under $f$ is a finite set.
\vspace{.1cm} 
\begin{thm}\label{thm:pth+powered+iterate} Let $f(x)=x^d+c$ for some nonzero $c\in\mathbb{Z}$ and some $d\geq2$. Moreover, assume that $f^N(\alpha)=\epsilon y^p$ for some $\alpha,y\in\mathbb{Z}$, some prime $p|d$, and some $\epsilon=\pm{1}$, where $N=4$ if $d=2$ and $N=3$ when $d\geq3$. Then $\alpha$ is preperiodic point for $f$ and $\epsilon y^p$ is a periodic point for $f$. More specifically, the following statements hold:\vspace{.1cm}  
\begin{enumerate}
\item[\textup{(1)}] If $d=2$, then $\alpha=\pm{\epsilon}y^2$ and $\epsilon y^2$ is a fixed point or point of exact period $2$ for $f$.   \vspace{.2cm}  
\item[\textup{(2)}] If $d\geq3$ is odd, then $\alpha=\epsilon y^p$ is a fixed point for $f$.\vspace{.2cm}   
\item[\textup{(3)}] If $d\geq4$ is even and $c\neq-1$, then $\alpha=\pm{\epsilon y^p}$ and $\epsilon y^p$ is a fixed point for $f$. \vspace{.2cm} 
\item[\textup{(4)}] If $d\geq4$ is even and $c=-1$, then $\alpha=\pm{\epsilon y^p}$ and $\epsilon y^p\in\{0,-1\}$ is a point of exact period $2$ for $f$.  
\end{enumerate}
\end{thm}
\begin{remark} We note that Theorem \ref{thm:pth+powered+iterate} is an improvement on Theorems 2.3 and 2.7 of \cite{hindes2025proportion} in two ways: we do not assume that the polynomial $f$ is irreducible nor do we assume that $f$ has prime degree. Likewise, Theorem \ref{thm:pth+powered+iterate} may be viewed as an unconditional and explicit version of \cite[Proposition 5.7]{PreperiodicPointsandABC} in the special case of rational integers.     
\end{remark}
To prove this result, we make several observations. 
\begin{lem}\label{lem:consecutive+pth+powers} Let $x>1$ and $d\geq2$. Then $x^d-(x-1)^d>x^{d-1}$.     
\end{lem}
\begin{proof} Let $y=x-1$. Then $x,y>0$, $x-y=1$, and 
\begin{equation*}
\begin{split}
x^d-(x-1)^d=x^d-y^d&=(x-y)(x^{d-1}+xy^{d-2}+\dots+y^{d-2}x+y^{d-1})\\[5pt]
&=(x^{d-1}+yx^{d-2}+\dots+y^{d-2}x+y^{d-1})>x^{d-1}
\end{split} 
\end{equation*}
as claimed. 
\end{proof}
Next, we note that if $f(\beta)=\epsilon y^p$ for some $\epsilon$, $p$ and $y$ as in the statement of Theorem \ref{thm:pth+powered+iterate} and some $\beta\in\mathbb{Z}$, then $|\beta|$ is not too large compared to $|c|$.
\begin{lem}\label{lem:pth+power+bound} Let $c\in\mathbb{Z}$ be nonzero, let $d\geq2$, and suppose that $\alpha^d+c=\epsilon y^p$ for some $\alpha,y\in\mathbb{Z}$, some prime $p|d$, and some $\epsilon\in\{\pm{1}\}$. Then $|\alpha|\leq \sqrt{|c|}$ if $d>2$ and $|\alpha|\leq |c|$ if $d=2$.
\end{lem}
\begin{proof}
First, consider the case where $p$ is odd. Notice that we may write $b=\alpha^{d/p}$ so that $b^p+c=\alpha^d+c=(\epsilon y)^p$. We claim that $|b| \le \sqrt{|c|}$. Since $|b| = |a^{d/p}| \ge |\alpha|$, it suffices to prove this claim. Rearranging the given equation, we have that $c = (\epsilon y)^p - b^p$. If $|b| = 1$, then we immediately have $|b| = 1 \le \sqrt{|c|}$, so we assume that $|b|\geq2$. It is clear that $|c|$ is minimized when $b$ and $\epsilon y$ are consecutive, so that $|c| \ge \min(2^p-1, 3^p-2^p) = 2^p - 1$. On the other hand, \cite[Lemma 2.8]{hindes2025proportion} implies that $|b|\leq(|c|/p)^{1/(p-1)}+1$. Moreover, it is straightforward to check that $(|c|/p)^{1/(p-1)}+1 \le \sqrt{|c|}$ when $|c| \ge 2^p-1$; thus, the desired result follows.

Now, we consider the case when $p=2$. In the same vein as above, we let $b = \alpha^{d/p}$, so that $b^2 + c = \epsilon y^2$. If $\epsilon = 1$, then \cite[Lemma 2.8]{hindes2025proportion} implies that $|b|\leq |c|/2 + 1$. Then since $\lfloor|c|/2 + 1\rfloor \le |c|$, we see that $|b| \leq |c|$, which is sufficient when $d = 2$. Likewise, when $d \ge 4$, we have that $|\alpha| \le |c|^{d/p} \le \sqrt{|c|}$ as claimed. 

Finally, suppose that $\epsilon = -1$. Then $0 \le b^2 = -y^2 - c \le -c$. But this immediately implies that $|b^2| \le -c$ and so $|b| \le \sqrt{|c|}$ as desired. 
\end{proof}
\begin{lem}\label{lem:compare+to+root+generic} Let $f(x)=x^d+c$ for some $c\in\mathbb{R}$ and $d\geq2$ and let $\rho:=\sqrt[d]{|c|}$. Moreover, assume that $\beta\in\mathbb{R}$ satisfies $||\beta|-\rho|\geq1$. Then the following statements hold:\vspace{.05cm}  
\begin{enumerate}
\item[\textup{(1)}] If $d>2$ and $|c|\geq2$, then $|f^n(\beta)|>\rho^{d-1}>\sqrt{|c|}$ for all $n\geq1$.\vspace{.2cm}  
\item[\textup{(2)}] If $d=2$ and $|c|\geq 3$, then $|f^n(\beta)|>|c|$ for all $n\geq2$. \vspace{.05cm}  
\end{enumerate}
\end{lem}
\begin{proof} When $d\geq3$, we assume that $|c|\geq2$, and when $d=2$, we assume that $|c|\geq3$. 

If $|\beta| -\rho\geq1$ and $d\geq3$, then Lemma \ref{lem:consecutive+pth+powers} implies that 
\[
|f(\beta)|=|\beta^d+c|\geq |\beta|^d-|c|=|\beta|^d-\rho^d\geq(\rho+1)^d-\rho^d>(\rho+1)^{d-1}.
\]
Moreover, since $(\rho+1)^{d-1}\geq\rho+1$, we may continue on in this way inductively to deduce that $|f^n(\beta)|>(\rho+1)^{d-1}$ for all $n\geq1$. In particular, statement (1) holds in this case since $(\rho+1)^{d-1}>\rho^{d-1}> \rho^{d/2}=\sqrt{|c|}$. 

Likewise, when $d=2$ and $|\beta|-1\geq \rho$, we have that $|f(\beta)|\geq (\rho+1)^2-\rho^2=2\rho+1>2\rho$. But then $|f^2(\beta)|\geq |f(\beta)|^2-|c|\geq (2\rho)^2-\rho^2=3\rho^2>|c|$. Moreover, since $3\rho^2>2\rho$, we may repeat this argument inductively to deduce that $|f^n(\beta)|>|c|$ for all $n\geq2$ as claimed.     

Now assume that $\rho-|\beta|\geq1$ and $d\geq3$. Then Lemma \ref{lem:consecutive+pth+powers} implies that
\[|f(\beta)|=|\beta^d+c|\geq |c|-|\beta|^d=\rho^d-|\beta|^d \geq \rho^d-(\rho-1)^d>\rho^{d-1}.\]
In particular, if $d\geq3$ then $|f(\beta)|>\rho^{d-1}\geq\rho^{d/2}=\sqrt{|c|}$. Furthermore, since $|f(\beta)|>\rho^{d-1}$, we compute that
\[|f^2(\beta)|\geq |f(\beta)|^d-|c|\geq\rho^{d(d-1)}-|c|=|c|^{d-1}-|c|\geq|c|^2-|c|\geq|c|.\]
Moreover, since $|f^2(\beta)|\geq |c|=\rho^d>\rho^{d-1}$, we may repeat the above argument inductively to deduce that  $|f^n(\beta)|\geq|c|>\sqrt{|c|}$ for all $n\geq2$. Hence, statement (1) holds in this case.

Finally, assume that $d=2$ and $\rho-|\beta|\geq1$. Then 
\[|f(\beta)|\geq|c|-|\beta|^2\geq \rho^2-(\rho-1)^2=2\rho-1.\]
In particular, computing one more iterate, we see that 
\[|f^2(\beta)|\geq |f(\beta)|^2-|c|\geq(2\rho-1)^2-\rho^2=(2\sqrt{c}-1)^2-|c|.\]
But it is straightforward to check that $(2\sqrt{c}-1)^2-|c|>|c|$ for all $|c|\geq3$. Hence, we deduce that $|f^2(\beta)|>|c|$. But $|c|=\rho^2>2\rho-1$ for all $\rho>1$, so we may repeat the above argument inductively to conclude that $|f^n(\beta)|>|c|$ for all $n\geq2$ as claimed.   
\end{proof}
We now have the tools in place to prove Theorem \ref{thm:pth+powered+iterate}, namely, that if a sufficiently large iterate of $f=x^d+c$ at $x=\alpha$ produces a $p$th power for some $p|d$, then $\alpha$ must be preperiodic for $f$. However, before we begin the technical details of the proof, we include the following sketch to aid the reader. Assume, for simplicity, that $f^6(\alpha)=\epsilon y^p$, that $d>2$, and that $|c|\geq3$. Then $f^5(\alpha)^d+c=f^6(\alpha)=\epsilon y^p$, so that Lemma \ref{lem:pth+power+bound} implies that $|f^5(\alpha)|\leq \sqrt{|c|}$. On the other hand, Lemma \ref{lem:compare+to+root+generic} then implies that the previous iterates $\alpha,f(\alpha),f^2(\alpha),f^3(\alpha),f^4(\alpha)$ are all contained in the set 
\[I_\rho:=\{b\in\mathbb{Z}\,: ||b|-\rho|<1\},\]
where $\rho=\sqrt[d]{|c|}>0$. But $I_\rho$ has at most $4$ elements, so we deduce that $f^n(\alpha)=f^m(\alpha)$ for some $0\leq n<m\leq4$ by the pigeon-hole principle. From here, the explicit descriptions of $\alpha$ and $\epsilon y^p$ given in Theorem \ref{thm:pth+powered+iterate} follow from classification results for the set of preperiodic points of $x^d+c$ when $c\in\mathbb{Z}$. In particular, the proof we give below carefully refines this sketch by including small values of $c$ and $d$ and by decreasing $N$ to $3$ or $4$, when appropriate.          
\begin{proof}[(Proof of Theorem \ref{thm:pth+powered+iterate})] We begin with some notation. Let $c\in\mathbb{Z}$ be nonzero, let $d\geq2$, and let $\rho=\sqrt[d]{|c|}$. Likewise, let $N=3$ when $d>2$ and $N=4$ when $d=2$. Finally, assume that $f^N(\alpha)=\epsilon y^p$ for some $\alpha,y\in\mathbb{Z}$, some $\epsilon=\pm{1}$, and some prime $p|d$. 

We first prove the result in the case when $c$ is sufficiently large. With this in mind, we assume that $|c|\geq2$ when $d>2$ and that $|c|\geq3$ when $d=2$. Furthermore, set $B(c)=\sqrt{|c|}$ when $d>2$ and $B(c)=|c|$ when $d=2$. In particular, with these hypotheses, we note that if either $||\alpha|-\rho|\geq1$ or $||f(\alpha)|-\rho|\geq1$, then it follows from Lemma \ref{lem:compare+to+root+generic} that $|f^{N-1}(\alpha)|>B(c)$. On the other hand,  since $f^{N-1}(\alpha)^d+c=f^N(\alpha)=\epsilon y^p$, Lemma \ref{lem:pth+power+bound} implies that $|f^{N-1}(\alpha)|\leq B(c)$, and we reach a contradiction. Hence, it must be the case that $|\alpha|$ and $|f(\alpha)|$ are both a distance \emph{strictly less than $1$} away from $\rho$. In particular, since $\alpha$ and $f(\alpha)$ are integers, we deduce that $||\alpha|-|f(\alpha)||\leq1$. From here we proceed in cases.  
\\[5pt]
\textbf{Case(1):} Suppose that $|\alpha|=|f(\alpha)|$ and that $d>2$. Then, if $d$ is even, we have that $f(\alpha)=f(|\alpha|)=f(|f(\alpha)|)=f(f(\alpha))$ and so $f(\alpha)$ is a fixed point of $f$. Hence, it follows that $\epsilon y^p=f^N(\alpha)=f^{N-1}(f(\alpha))=f(\alpha)$. Thus, $\epsilon y^p$ is a fixed point of $f$ and $\alpha= \pm{f(\alpha)}=\pm{\epsilon y^p}$. In particular, we obtain the description of $\alpha$ and $\epsilon y^p$ given in statement (3). Now assume that $d$ is odd. If $\alpha=f(\alpha)$, then $\alpha=f^3(\alpha)=\epsilon y^p$ is a fixed point of $f$, which fits the description in statement (2). On the other hand, if $f(\alpha)=-\alpha$, then $c=-\alpha^d-\alpha=-(\alpha^d+\alpha)$. In particular, it follows that $|c|=|\alpha|^d+|\alpha|$. Moreover, we compute that 
\[f^2(\alpha)=f(f(\alpha))=f(-\alpha)=(-\alpha)^d+c=-\alpha^d-\alpha^d-\alpha=-(2\alpha^d+\alpha).\]
Therefore, we deduce that $|f^2(\alpha)|=2|\alpha|^d+|\alpha|=|\alpha|^d+|c|$. But then, $|f^2(\alpha)|>|c|\geq \sqrt{|c|}$, which contradicts Lemma \ref{lem:pth+power+bound} and the fact that $f^2(\alpha)^d+c=f^3(\alpha)=\epsilon y^p$ by assumption; here we use also that $\alpha\neq0$, since $\rho>1$ and $||\alpha|-\rho|<1$. \\[5pt] 
\textbf{Case(2):} Suppose that $|f(\alpha)|=|\alpha|+1$ and that $d>2$. Then $|\alpha|+1=|f(\alpha)|=|\alpha^d+c|\geq |c|-|\alpha|^d$ so that $|\alpha|^d+|\alpha|+1\geq |c|$. Hence, for $d>2$ we see that   
\begin{equation}\label{eq:even+2nd+iterate1}
\begin{split}
|f^2(\alpha)|&\geq |f(\alpha)|^d-|c|\geq(|\alpha|+1)^d-(|\alpha|^d+|\alpha|+1) \\[5pt] 
&\geq d|\alpha|^{d-1}+(d-1)|\alpha|>|\alpha|^{d/2}+|\alpha|+1.\\[5pt]
&\geq \sqrt{|\alpha|^d}+\sqrt{|\alpha|}+\sqrt{1}\geq \sqrt{|\alpha|^d+|\alpha|+1}\\[5pt]
&\geq\sqrt{|c|}. 
\end{split}
\end{equation}
Here we use that $\alpha\neq0$, since $\rho>1$ and $|\alpha-\rho|<1$. However, the bound in \eqref{eq:even+2nd+iterate1} contradicts Lemma \ref{lem:pth+power+bound} and the fact that $f^2(\alpha)^d+c=f^3(\alpha)=\epsilon y^p$ by assumption. \\[5pt]
\textbf{Case(3):} Suppose that $|f(\alpha)|=|\alpha|-1$ and that $d>2$. Then $|\alpha|-1=|f(\alpha)|=|\alpha^d+c|\geq |\alpha|^d-|c|$ so that $|c|\geq |\alpha|^d-|\alpha|+1$. Similarly, $|\alpha|-1=|f(\alpha)|\geq |c|-|\alpha|^d$ so that $|c|\leq|\alpha|^d+|\alpha|-1$. Hence, for $d>2$ we see that
\begin{equation}\label{eq:even+2nd+iterate2}
\begin{split}
|f^2(\alpha)|&=|f(\alpha)^d+c|\geq |c|-|f(\alpha)|^d\geq (|\alpha|^d-|\alpha|+1)-(|\alpha|-1)^d \\[5pt] 
&=\big(|\alpha|^d-(|\alpha|-1)^d\big)-|\alpha|+1  \\[5pt]
&>|\alpha|^{d-1}-|\alpha|+1\geq\sqrt{|\alpha|^d+|\alpha|-1}\\[5pt]
&\geq \sqrt{|c|}.
\end{split}
\end{equation}
Here we use Lemma \ref{lem:consecutive+pth+powers} applied to $\beta=|\alpha|$ and the fact that $|\alpha|\geq2$ (since $|c|\geq2$ by assumption and $|c|\leq|\alpha|^d+|\alpha|-1$ by above).  However, the bound in \eqref{eq:even+2nd+iterate2} contradicts Lemma \ref{lem:pth+power+bound} and the fact that $f^2(\alpha)^d+c=f^3(\alpha)=\epsilon y^p$ by assumption. \\[5pt]
\textbf{Case(4):} Suppose that $d=2$ (so $p=2$ also). As in Case (1), if $|\alpha|=|f(\alpha)|$, then one may check that $\epsilon y^2$ is a fixed point of $f$ and $\alpha=\pm{\epsilon y^2}$. Hence, we obtain one of the descriptions in statement (1). 

Now suppose that $|f(\alpha)|=|\alpha|+1$. Note that it must be the case that $|\alpha|<\rho$, since otherwise we contradict the fact that $|f(\alpha)-\rho|<1$. But then $f(\alpha)=f(|\alpha|)=|\alpha|^2-\rho^2<0$ and thus 
\[|\alpha|+1=|f(\alpha)|=-f(\alpha)=-f(|\alpha|)=-|\alpha|^2-c.\]
In particular, we see that $c=-|\alpha|^2-|\alpha|-1$ and so \vspace{.05cm} 
\[f^2(|\alpha|)=f^2(\alpha)=f(|f(\alpha)|)=f(|\alpha|+1)=(|\alpha|+1)^2+c=(|\alpha|+1)^2-(|\alpha|^2+|\alpha|+1)=|\alpha|.
\vspace{.05cm} 
\]
Hence, $|\alpha|$ a is periodic point of exact period $2$; note that $f(|\alpha|)\neq |\alpha|$, since we have shown $f(|\alpha|)=f(\alpha)<0$. In particular, $\epsilon y^2=f^4(\alpha)=f^4(|\alpha|)=f^2(f^2(|\alpha|))=|\alpha|$ is a point of exact period $2$ and $\alpha=\pm{\epsilon y^2}$, which fits a description in statement (1) of Theorem \ref{thm:pth+powered+iterate}. 

Now suppose that $|f(\alpha)|=|\alpha|-1$. Note that it must be the case that $|\alpha|>\rho$, since otherwise we contradict the fact that $|f(\alpha)-\rho|<1$. But then $f(\alpha)=f(|\alpha|)=|\alpha|^2-\rho^2>0$ and thus 
\[|\alpha|-1=|f(\alpha)|=f(\alpha)=f(|\alpha|)=|\alpha|^2+c.\]
In particular, we deduce that $c=-|\alpha|^2+|\alpha|-1$ and so \vspace{.05cm}
\[f^2(-|\alpha|)=f(|f(\alpha)|)=f(|\alpha|-1)=(|\alpha|-1)^2+c=(|\alpha|-1)^2+(-|\alpha|^2+|\alpha|-1)=-|\alpha|. \vspace{.05cm}\] 
Hence, $-|\alpha|$ a is periodic point of exact period $2$; note that $f(-|\alpha|)\neq -|\alpha|$, since we have shown $f(-|\alpha|)=f(\alpha)>0$. In particular, $\epsilon y^2=f^4(\alpha)=f^4(-|\alpha|)=f^2(f^2(-|\alpha|))=-|\alpha|$ has period $2$ and $\alpha=\pm{\epsilon y^2}$, which fits a description in statement (1) of Theorem \ref{thm:pth+powered+iterate}.\\[5pt] 
Therefore, it remains to consider the case of small constant terms. Specifically, it remains to consider $c=\pm{1},\pm{2}$ when $d=2$ and $c=\pm{1}$ when $d>2$. However, when $d=2$ and $c\in\{1,2\}$, we check with {\tt{Magma}} \cite{Magma} that the equation $f^4(\alpha)=\epsilon y^2$ has no solutions $\alpha,y\in\mathbb{Z}/8\mathbb{Z}$ and $\epsilon=\pm{1}$. Hence, there are no solutions with $\alpha,y\in\mathbb{Z}$, so Theorem \ref{thm:pth+powered+iterate} is vacuously true in these cases. Now when $f(x)=x^2-2$, if $f^4(\alpha)=\epsilon y^2$, then Lemma \ref{lem:pth+power+bound} implies that $f^3(\alpha)\in\{0,\pm{2},\pm{1}\}$. Moreover, repeatedly computing rational preimages $f^{-1}(b)$ for $b\in\{0,\pm{1},\pm{2}\}$, we see that $\alpha\in \{0,\pm{1},\pm{2}\}$. On the other hand, $f^4(0)=f^4(\pm{2})=2$, and $2$ is not of the form $\epsilon y^2$. Thus, it must be the case that $\alpha=\pm{1}$ and $\epsilon y^2=-1$ when $f=x^2-2$, which fits a description in statement (1). Similarly, if $f=x^2-1$, then Lemma \ref{lem:pth+power+bound} implies that  $f^3(\alpha)\in\{0,\pm{1}\}$. Moreover, repeatedly computing rational preimages $f^{-1}(b)$ for $b\in\{0,\pm{1}\}$, we see that $\alpha\in \{0,\pm{1}\}$. Thus, $\epsilon y^2=f^4(\pm{1})=-1$ or $\epsilon y^2=f^4(0)=0$. Hence in either case, $\epsilon y^2$ is a periodic point of exact period $2$ for $f$ and $\alpha=\pm{\epsilon y^2}$, which fits a description in statement (1). 

Finally, we consider the case when $c=\pm{1}$ and $d>2$. If $c=-1$ and $d$ is even, then the the same argument given above for $f=x^2-1$ yields statement (4). Now suppose that $c=-1$ and $d$ is odd. Then Lemma \ref{lem:pth+power+bound} implies that $f^2(\alpha)\in\{0,\pm{1}\}$. Moreover, repeatedly computing rational preimages $f^{-1}(b)$ for $b\in\{0,\pm{1}\}$, we see that $\alpha=1$. However, $f^3(1)=-2$ in this case, which is not of the form $\epsilon y^p$ for any prime $p$ and any $y\in\mathbb{Z}$, and we reach a contradiction. Finally, suppose that $c=1$ and $d>2$. Again, Lemma \ref{lem:pth+power+bound} implies that $f^2(\alpha)\in\{0,\pm{1}\}$. However, if $d$ is even, then $f^2(\alpha)>1$ for all $\alpha\in\mathbb{R}$, and we reach a contradiction. If $d$ is odd, then repeatedly computing rational preimages $f^{-1}(b)$ for $b\in\{0,\pm{1}\}$, we see that $\alpha=1$. However, $f^3(1)=2$ in this case, which is not of the form $\epsilon y^p$ for any prime $p$ and any $y\in\mathbb{Z}$, and we reach a contradiction. This completes the proof of the theorem.                  
\end{proof}
\begin{remark}\label{rem:pth+powers+implies+intro+version} Note that Theorem \ref{thm:pth+powered+iterate} implies Theorem \ref{thm:pth+powered+iterate+intro}. Namely, if $f^n(\alpha)=\epsilon y^p$ for some $\alpha$, $\epsilon$, and $y$ as in Theorem \ref{thm:pth+powered+iterate+intro} and some $n\geq N$, where $N$ is defined in Theorem \ref{thm:pth+powered+iterate}, then $f^N(f^{n-N}(\alpha))=\epsilon y^p$. Hence, Theorem \ref{thm:pth+powered+iterate} implies that $f^{n-N}(\alpha)$ is preperiodic and $\epsilon y^p$ is periodic for $f$ respectively. But then $\alpha$ is also preperiodic for $f$ as claimed.    
\end{remark}
\section{Irreducible polynomials in semigroups}\label{sec:irreduciblity}
We now apply Theorem \ref{thm:pth+powered+iterate}, on the classification of $p$th powered images, to construct irreducible polynomials in unicritically generated semigroups. As a first step, we have the following link between reducible polynomials with a compositional factor of the form $x^d+c$ and $p$th powers.
\begin{prop}\label{prop:irreducibility+test} Let $K$ be a field of characteristic zero, let $w(x)\in K[x]$ be monic and irreducible, and let $u(x)=x^d+c$ for some $c\in K$ and $d\geq2$. Moreover, if $d$ is even, assume that $w$ has even degree.  Then $w\circ u$ is irreducible over $K$ unless $w(u(0))=y^p$ for some $y\in K$ and some prime $p|d$.
\end{prop}
\begin{proof} 
Let $w$ and $u$ be as above and assume that $w\circ u$ is reducible over $K$. Then Capelli's Lemma implies that $u(x)-\alpha=x^d+c-\alpha$ is reducible over $K(\alpha)$ for some root $\alpha\in\overline{K}$ of $w$. From here, \cite[Theorem 9.1, p. 297]{MR1878556} implies that $\alpha-c=z^p$ for some $z\in K(\alpha)$ and some prime $p|d$ or $\alpha-c=-4z^4=-(2z^2)^2$ when $4|d$. In particular, it must be the case that either $\alpha-c=z^p$ for some odd $p$ and $z\in K(\alpha)$ or $\alpha-c=\pm{z^2}$ for some $z\in K(\alpha)$. On the other hand, since $w\in K[x]$ is irreducible and the norm map $N_{K(\alpha)/K}: K(\alpha)\rightarrow K$ is multiplicative, we have that  \vspace{.05cm}
\begin{equation}\label{eq:norm}
\begin{split} 
\scalemath{.95}{
N_{K(\alpha)/K}(\alpha-c)=(-1)^{\deg(w)}N_{K(\alpha)/K}(c-\alpha)=(-1)^{\deg(w)}w(c)=(-1)^{\deg(w)}w(u(0)).}
\vspace{.05cm}
\end{split} 
\end{equation}
Therefore, if $\alpha-c=z^p$, then we have that $(-1)^{\deg(w)}y^p=w(u(0))$ where $y=N_{K(\alpha)/K}(z)\in K$. In particular, if $p$ is odd, then $w(u(0))$ must be a $p$th power in $K$: when $\deg(w)$ is even, $w(u(0))=y^p$, and when $\deg(w)$ is odd, $w(u(0))=(-y)^p$. On the other hand, when $p=2$ we assume that $\deg(w)$ is even. Hence, if $\alpha-c=\pm{z}^2$, then \eqref{eq:norm} implies that
\vspace{.1cm} 
\[
\scalemath{.93}{
y^2=\pm{1}^{\deg(w)}N_{K(\alpha)/K}(z)^2=N_{K(\alpha)/K}(\pm{z}^2)=N_{K(\alpha)/K}(\alpha-c)=(-1)^{\deg(w)}w(u(0))=w(u(0)).}
\vspace{.1cm} 
\]
Therefore, we deduce in all cases that $w(u(0))$ is a $p$th power in $K$ for some prime $p|d$. 
\end{proof}
As a consequence, we note that if $f=x^d+c$ for some $c\in\mathbb{Z}$ is irreducible over $\mathbb{Q}$, then $f^n$ is irreducible over $\mathbb{Q}$ for all $n\geq1$. Equivalently, $f$ is stable over $\mathbb{Q}$ if and only if $f$ is irreducible over $\mathbb{Q}$; compare to similar results in \cite{PreperiodicPointsandABC,jones2008density}.
\begin{prop}\label{prop:stability} Let $f(x)=x^d+c$ for some $c\in\mathbb{Z}$ and $d\geq2$. If $f$ is irreducible over $\mathbb{Q}[x]$, then $f^n$ is irreducible over $\mathbb{Q}[x]$ for all $n\geq1$.   
\end{prop}
We begin with the following, very simple, yet useful observation. 
\begin{lem}\label{lem:oribt+zero} Let $f(x)=x^d+c$ for some $|c|\geq2$ and $d\geq2$. Then $|f^m(0)|\geq|c|$ for all $m\geq1$.      
\end{lem}
\begin{proof} The claim is obvious when $m=1$. On the other hand, if $|f^k(0)|\geq|c|$ for some $k\geq0$, then we see that 
\[|f^{k+1}(0)|=|(f^k(0))^d+c|\geq|f^k(0)|^d-|c|\geq |c|^{d}-|c|=|c|(|c|^{d-1}-1)\geq|c|(2^{2-1}-1)=|c|.\]
Therefore, it follows by induction that $|f^m(0)|\geq|c|$ for all $m\geq1$ as claimed.   
\end{proof}
\begin{proof}[(Proof of Proposition \ref{prop:stability})] The $d=2$ case follows from \cite[Proposition 4.5]{jones2008density}. Assume that $d\geq3$, that $f$ is irreducible over $\mathbb{Q}$, and that $f^n$ is reducible over $\mathbb{Q}$ for some $n>1$. Moreover, we may assume that $n$ is the minimum iterate with this property. Then Proposition \ref{prop:irreducibility+test} implies that $f^n(0)=y^p$ for some $y\in\mathbb{Z}$ and some prime $p|d$. Then Lemma \ref{lem:pth+power+bound} implies that $|f^{n-1}(0)|\leq \sqrt{|c|}$. However, if $|c|\geq2$, then Lemma \ref{lem:oribt+zero} implies that $|f^{n-1}(0)|\geq|c|$. Hence, $|c|\leq |f^{n-1}(0)|\leq \sqrt{|c|}$, and we obtain a contradiction. Therefore, it must be the case that $|c|\leq1$. However, $f=x^d$ and $f=x^d-1$ are both reducible over $\mathbb{Q}$ and were therefore excluded at the outset. Hence, it suffices to consider $f=x^d+1$. But in this case, it is clear that $f^m(0)>1$ for all $m>1$. In particular, since we still have that $f^n(0)=y^p$, Lemma \ref{lem:pth+power+bound} implies that $f^{n-1}(0)\leq 1$. Thus $n=2$ and $2=f^2(0)=y^p$, a contradiction.          
\end{proof}
Next, to simplify the statements of some results below, we make the following definition.
\begin{defin} Let $K$ be a field and let $f(x)=x^d+c$ for some $c\in K$. Then we say that $f$ contains a \emph{powered fixed} if there exists $y\in K$ and a prime $p|d$ such that $f(y^p)=y^p$. Likewise, we say that $f$ contains a \emph{powered $2$-cycle} if $f(f(y^p))=y^p$ for some $y$ and $p$ with $f(y^p)\neq y^p$.        
\end{defin}
In particular, combining the stability result above with the $p$th powered classification result from the introduction, we prove the following result. Namely, if a semigroup of the form $\langle x^d+c_1,\dots, x^d+c_s\rangle$ contains an irreducible polynomial without a powered fixed point (or 2-cycle when $d=2$), then it contains many irreducible polynomials.  
\begin{prop}\label{prop:no+special+irreducibles} Let $G=\langle x^d+c_1,\dots, x^d+c_s\rangle$ for some $d \ge 2$ and some $c_1,\dots,c_s\in\mathbb{Z}$. Moreover, assume that $G$ contains an irreducible polynomial. Then there is an irreducible $f(x)=x^d+c$ in $G$ and the following statements hold: \vspace{.1cm}  
\begin{enumerate}
    \item[\textup{(1)}] If $d=2$, then either $\{f^4\circ g\,:\, g\in G\}$ is a set of irreducible polynomials in $\mathbb{Q}[x]$ or $f$ has a powered fixed point or powered $2$-cycle.\vspace{.2cm}
    \item[\textup{(2)}] If $d\geq3$, either $\{f^3\circ g\,:\, g\in G\}$ is a set of irreducible polynomials in $\mathbb{Q}[x]$ or $f$ has a powered fixed point.\vspace{.1cm} 
\end{enumerate} 
\end{prop}
\begin{proof} Let $N=4$ when $d=2$ and $N=3$ otherwise. If  $G=\langle x^d+c_1,\dots, x^d+c_s\rangle$ contains an irreducible polynomial, then some generator $x^d+c_i$ must be irreducible. Let $f(x)=x^d+c$ be such an element. Then Proposition \ref{prop:stability} implies that $f^N$ is irreducible also. Now assume that $f^N\circ g$ is reducible for some $g\in G$. Then repeated application of Proposition \ref{prop:irreducibility+test} implies that $f^N(\alpha)=y^p$ for some $\alpha,y\in\mathbb{Z}$ and some prime $p|d$. But then Theorem \ref{thm:pth+powered+iterate} implies that $y^p$ must be a fixed point when $d\geq3$, and a fixed point or point of exact period $2$ when $d=2$.           
\end{proof}
Moreover, leveraging the fact that two maps of the form $x^d+c$ with the same fixed point must be equal, we deduce the following result in the odd degree case. 
\begin{remark}\label{rem:fixed+pts+unique} We note that the fixed points of $f(x)=x^d+c$ for nonzero $c\in\mathbb{Z}$ and $d\geq3$ are unique: if $y,z\in\mathbb{Z}$ satisfy $y-y^d=z-z^d$, then reordering gives $z^d+(y-z)=y^d$. Now assume for a contradiction that $y\neq z$. Then Lemma \ref{lem:pth+power+bound} applied to both $z$ and $y$ implies that that $\max\{|y|,|z|\}\leq \sqrt{2\max\{|y|,|z|\}}$. However, this bound forces $\max\{|y|,|z|\}\leq2$. On the other hand note that $y,z\not\in\{0,1\}$ since $c\neq0$. Assume next that $y=-1$, then since $z\neq y$ we must have that $|z|=2$. Moreover, $d$ must be even since otherwise $c=0$. But then $2=|-2|=|y-y^d|=|z||1-z^{d-1}|$ and so $|1-z^{d-1}|=1$. However, this forces $z^{d-1}=0,2$, which is a contradiction since $|z|=2$ and $d-1>1$. Moreover, by symmetry we have that $z\neq-1$. Therefore, it must be the case that $|y|=|z|=2$ and $y\neq z$. Without loss, say $y=2$ and $z=-2$. However, in this case we deduce from the fact that $y-y^d=z-z^d$ that $-2=2$ when $d$ is even and $1=2^{d-1}$ when $d$ is odd. In either case, we reach a contradiction and deduce that $y=z$ as claimed.     
\end{remark}
\begin{prop}\label{prop:special+irreducible+odd} Let $G=\langle x^d+c_1,\dots, x^d+c_s\rangle$ for some odd $d \ge 3$ and some $c_1,\dots,c_s\in\mathbb{Z}$. Moreover, assume that $G$ contains an irreducible polynomial $f_1(x)=x^d+c_1$ with a powered fixed point $y^p$ and another polynomial $f_2(x)=x^d+c_2$ such that $c_2\neq y^p$. Then the following statements hold:
\vspace{.1cm}  
\begin{enumerate}
    \item[\textup{(1)}] If $c_2=0$, then $\{f_1^3\circ f_2\circ f_1\circ g\,:\, g\in G\}$ 
    is a set of irreducible polynomials.    \vspace{.2cm}
    \item[\textup{(2)}] If $c_2\neq 0$, then $\{f_1^3\circ f_2^3\circ g\,:\, g\in G\}$ 
    is a set of irreducible polynomials. \vspace{.2cm}
\end{enumerate} 
\end{prop}
\begin{proof} Assume that $d\geq3$ is odd and that $G$ contains an irreducible polynomial $f_1(x)=x^d+c_1$ with a powered fixed point $y^p$ and another polynomial $f_2(x)=x^d+c_2$ such that $c_2\neq y^p$. Then $f_1^3$ is is irreducible by Proposition \ref{prop:stability}. Hence, Proposition \ref{prop:irreducibility+test} implies that $f_1^3\circ f_2$ is irreducible unless $f_1^3(f_2(0))=z^q$ for some $z\in\mathbb{Z}$ and some prime $q|d$. However, in the latter case, Theorem \ref{thm:pth+powered+iterate} implies that $c_2=f_2(0)=z^q$ is a fixed point of $f_1$. But $y^p$ is a fixed point of $f_1$ by assumption and fixed points are unique by Remark \ref{rem:fixed+pts+unique}. Thus, $c_2=y^p$, and we reach a contradiction. In particular, we deduce that $f_1^3\circ f_2$ is irreducible.       

From here, we first handle the case when $c_2\neq0$. Note that if $f_1^3\circ f_2^2$ is reducible, then $f_1^3(f_2^2(0))=z^q$ by Proposition \ref{prop:irreducibility+test}. But then again we have $f_2^2(0)=z^q=y^p$ by Theorem \ref{thm:pth+powered+iterate} and Remark \ref{rem:fixed+pts+unique}. Hence, $f_2(0)^d+c_2=f_2^2(0)=y^p$, and so Lemma \ref{lem:pth+power+bound} implies that $|c_2|=|f_2(0)|\leq\sqrt{|c_2|}$. Hence, it must be the case that $c_2=\pm{1}$. However, when this is the case we have that $f_2^2(0)=\pm{2}$, which is not of the form $y^p$. Therefore, we deduce that $f_1^3\circ f_2^2$ is irreducible. Finally, if $f_1^3\circ f_2^3\circ g$ is reducible where $g\in G$ or $g(x)=x$, then repeated application of Proposition \ref{prop:irreducibility+test} implies that $f_1^3(f_2^3(\alpha))=z^q$ for some $\alpha,z\in\mathbb{Z}$ and some prime $q|d$. But again Theorem \ref{thm:pth+powered+iterate} and Remark \ref{rem:fixed+pts+unique} applied to the map $f=f_1$ together imply that $f_2^3(\alpha)=y^p$. However, this time applying Theorem \ref{thm:pth+powered+iterate} to $f=f_2$, we deduce that $y^p$ must also be a fixed point of $f_2$. Thus $c_1=y^p-y^{pd}=c_2$, and we reach a contradiction. Therefore, every polynomial of the form $f_1^3\circ f_2^3\circ g$ for $g\in G$ is irreducible as claimed.

Finally, assume that $c_2=0$ so that $f_2(x)=x^d$. Recall that we have already shown that $f_1^3\circ f_2$ is irreducible. Hence, if $f_1^3\circ f_2\circ f_1\circ g$ is reducible for some $g\in G$ or $g(x)=x$, then (as in previous cases) we see that Theorem \ref{thm:pth+powered+iterate}, Proposition \ref{prop:irreducibility+test}, and Remark \ref{rem:fixed+pts+unique} together imply that $f_1^3(f_2(f_1(\alpha)))=y^p$ and so $(\alpha^d+y^p-y^{pd})^d=f_2(f_1(\alpha))=y^p$ for some $\alpha\in\mathbb{Z}$. Now let $t:=\alpha^d+y^p-y^{pd}$ so that $t^d=y^p$ and
\[t=\alpha^d+y^p-y^{pd}=\alpha^d+t^d-t^{d^2}.\]
Rearranging terms gives $(t^d)^d+(t-t^d)=\alpha^d$, so that Lemma \ref{lem:pth+power+bound} implies that $|t^d|\leq \sqrt{|t-t^d|}$. Here we use that $|t|>1$, since otherwise $|y^p|\leq1$ which implies that $c_1=0$ and $f_1$ is reducible, a contradiction. On the other hand, it is straightforward to check that the bound $|t^d|\leq \sqrt{|t-t^d|}$ on $t\in\mathbb{Z}$ implies that $t=0$, and we again contradict our assumption that $f_1=x^d+c_1$ is irreducible. Therefore, $f_1^3\circ f_2\circ f_1\circ g$ is irreducible for all $g\in G$ as claimed.   
\end{proof}
Likewise, we prove an analogous result when $d>2$ is even.
\begin{prop}\label{prop:special+irreducible+even} Let $G=\langle x^d+c_1,\dots, x^d+c_s\rangle$ for some even $d \ge 4$ and some $c_1,\dots,c_s\in\mathbb{Z}$. Moreover, assume that $G$ contains an irreducible polynomial $f_1(x)=x^d+c_1$ with a powered fixed point $y^p$ and another polynomial $f_2(x)=x^d+c_2$ such that $c_2\not\in\{\pm{y^p},-y^p-y^{pd}\}$. Then the following statements hold:
\vspace{.1cm}  
\begin{enumerate}
    \item[\textup{(1)}] If $c_2\in\{0,-1\}$, then $\{f_1^3\circ f_2\circ f_1\circ g\,:\, g\in G\}$ 
    is a set of irreducible polynomials.    \vspace{.2cm}
    \item[\textup{(2)}] If $c_2\not\in\{0,-1\}$, then $\{f_1^3\circ f_2^3\circ g\,:\, g\in G\}$ 
    is a set of irreducible polynomials. \vspace{.2cm}
\end{enumerate} 
\end{prop}
\begin{proof} Assume that $d\geq4$ is even and that $G$ contains an irreducible polynomial $f_1(x)=x^d+c_1$ with a powered fixed point $y^p$ and another polynomial $f_2(x)=x^d+c_2$ such that $c_2\not\in\{\pm{y^p},-y^p-y^{pd}\}$. Then $f_1^3$ is is irreducible by Proposition \ref{prop:stability}. Hence, Proposition \ref{prop:irreducibility+test} implies that $f_1^3\circ f_2$ is irreducible unless $f_1^3(f_2(0))=z^q$ for some $z\in\mathbb{Z}$ and some prime $q|d$. However, in this case, Theorem \ref{thm:pth+powered+iterate} implies that $c_2=f_2(0)=\pm{z^q}$ and $z^q$ is a fixed point of $f_1$. But $y^p$ is also a fixed point of $f_1$ so that $z^q=y^p$ by Remark \ref{rem:fixed+pts+unique}. Thus, $c_2=\pm{y^p}$, and we reach a contradiction. In particular, it must be the case that $f_1^3\circ f_2$ is irreducible.

From here we consider the case when $c_2\not\in\{0,-1\}$. Note that if $f_1^3\circ f_2^2$ is reducible, then $f_1^3(f_2^2(0))=z^q$ by Proposition \ref{prop:irreducibility+test}. But then again, Theorem \ref{thm:pth+powered+iterate} and Remark \ref{rem:fixed+pts+unique} imply that $f_2^2(0)=\pm{y^p}$. Hence, $f_2(0)^d+c_2=f_2^2(0)=\pm{y^p}$, and so Lemma \ref{lem:pth+power+bound} implies that $|c_2|=|f_2(0)|\leq\sqrt{|c_2|}$. Hence, it must be the case that $|c_2|\leq1$ so that $c_2=1$. However, when this is the case we have that $f_2^2(0)=2$, which is not of the form $\pm{y^p}$. Therefore, we deduce that $f_1^3\circ f_2^2$ is irreducible. Finally, if $f_1^3\circ f_2^3\circ g$ is reducible where $g\in G$ or $g(x)=x$, then repeated application of Proposition \ref{prop:irreducibility+test} implies that $f_1^3(f_2^3(\alpha))=z^q$ for some $\alpha,z\in\mathbb{Z}$ and some prime $q|d$. However, again Theorem \ref{thm:pth+powered+iterate} and Remark \ref{rem:fixed+pts+unique} applied to the map $f=f_1$ together imply that $f_2^3(\alpha)=\pm{y^p}$. But then applying Theorem \ref{thm:pth+powered+iterate} to $f=f_2$, we deduce that $\pm{y^p}$ must be a fixed point of $f_2$. If $y^p$ is a fixed point of $f_2$, then $c_1=y^p-y^{pd}=c_2$, and we reach a contradiction. Likewise, if $-y^p$ is a fixed point of $c_2$, then $c_2=(-y^p)-(-y^p)^d=-y^p-y^{pd}$, and we again reach a contradiction. Therefore, every polynomial of the form $f_1^3\circ f_2^3\circ g$ for $g\in G$ is irreducible in this case as claimed.

Next assume that $c_2=0$ so that $f_2(x)=x^d$. Recall that we have already shown that $f_1^3\circ f_2$ is irreducible. Hence, if $f_1^3\circ f_2\circ f_1\circ g$ is reducible for some $g\in G$ or $g(x)=x$, then (as in previous cases) we see that Theorem \ref{thm:pth+powered+iterate}, Proposition \ref{prop:irreducibility+test}, and Remark \ref{rem:fixed+pts+unique} together imply that $f_1^3(f_2(f_1(\alpha)))=y^p$ and $(\alpha^d+y^p-y^{pd})^d=f_2(f_1(\alpha))=\pm{y^p}$ for some $\alpha\in\mathbb{Z}$. Now let $t:=\alpha^d+y^p-y^{pd}$ so that $\pm{t^d}=y^p$ and
\[t=\alpha^d+y^p-y^{pd}=\alpha^d\pm{t^d}-(\pm{t^{d}})^d=\alpha^d\pm{t^d}-t^{d^2}.\]
Rearranging terms gives $(t^d)^d+(t\pm{t^d})=\alpha^d$. Now, if $t\pm{t^d}=0$, then $t=0,\pm{1}$ and so $y=0,\pm{1}$. However, if $y=0,1$, then $f_1=x^d$ is reducible, a contradiction. Likewise, if $y=-1$ and $p=2$, then $f_1=x^d$. On the other hand, if $y=-1$ and $p$ is odd, then $f_1=x^d-2$. But we still have that $(\alpha^d-2)^d=f_2(f_1(\alpha))=\pm{y^p}=\pm{1}$. Moreover, since $d$ is even, $\alpha^d-2=1$ so that $\alpha^d=3$, and we obtain a contradiction. Therefore, we may assume that $|t|>1$ and that  $(t^d)^d+(t\pm{t^d})=\alpha^d$. But then, Lemma \ref{lem:pth+power+bound} implies that $|t|^d\leq \sqrt{|t\pm{t^d}|}\leq\sqrt{|t|+|t|^d}$, which contradicts the fact that $|t|\geq2$. Therefore, it must be the case that every polynomial of the form $f_1^3\circ f_2^3\circ g$ for $g\in G$ is irreducible in this case when $c_2=0$ as claimed.     

Finally, assume that $c_2=-1$ so that $f_2(x)=x^d-1$. We have already shown that $f_1^3\circ f_2$ is irreducible. Hence, if $f_1^3\circ f_2\circ f_1\circ g$ is reducible for some $g\in G$ or $g(x)=x$, then (as in previous cases) we see that Theorem \ref{thm:pth+powered+iterate}, Proposition \ref{prop:irreducibility+test}, and Remark \ref{rem:fixed+pts+unique} together imply that $f_1^3(f_2(f_1(\alpha)))=y^p$ and $f_2(f_1(\alpha))=\pm{y^p}$ for some $\alpha\in\mathbb{Z}$. But then Lemma \ref{lem:pth+power+bound} implies that $|f_1(\alpha)|\leq|c_2|=1$. In particular, since $f_1$ is irreducible (so has not rational roots), it must be the case that $\alpha^d+y^p-y^{pd}=f_1(\alpha)=\pm{1}$. Suppose that $|y|\geq2$. Then $\alpha^d=(y^{p})^d+(-y^p\pm{1})$ and so Lemma \ref{lem:oribt+zero} implies that $|y|^p\leq \sqrt{|y|^p+1}$, which contradicts that $|y|\geq2$. Hence, $y\in\{0,1,-1\}$. But when $y=0,1$, we have that $f_1=x^d$, contradicting the fact that $f_1$ is irreducible. Likewise, if $y=-1$ and $p=2$, then $f_1=x^d$. Therefore, it must be the case that $y=-1$, that $p$ is odd, and that $f_1=x^d-2$. Hence, $(\alpha^d-2)^d-1=\pm{y^p}=\pm{1}$. However, if $(\alpha^d-2)^d-1=1$, then $(\alpha^d-2)^d=2$ and we reach a contradiction. Likewise, if $(\alpha^d-2)^d-1=-1$, then $\alpha^d=2$ and we reach a contradiction. In particular, every polynomial of the form $f_1^3\circ f_2^3\circ g$ for $g\in G$ is irreducible when $c_2=-1$ as claimed.\end{proof}
Next we note that the semigroups we study in this paper are free; compare to a similar result in \cite[Theorem 3.1]{hindes2021orbit}, where the degrees of the maps in the generating set are allowed to be distinct and the constant terms are assumed to be nonzero.
\begin{prop}\label{prop:free} Let $K$ be a field of characteristic zero, let $d\geq2$, and let \[\mathcal{U}_d:=\big\langle\{x^d+c\,:c\in K\}\big\rangle\]
be the semigroup generated by all polynomials of the form $x^d+c$ for some $c\in K$. Then $\mathcal{U}_d$ is a free semigroup. In particular, every finitely generated semigroup $G=\langle x^d+c_1,\dots, x^d+c_s\rangle$ for some $c_1,\dots,c_s\in K$ is also a free semigroup.       
\end{prop}
\begin{proof} Suppose that $\theta_1\circ\dots\circ\theta_n=\tau_1\circ\dots\circ \tau_m$ for some unicritial polynomials $\theta_1,\dots,\tau_m\in K[x]$ all of the same degree $d\geq2$ and some $n,m\geq1$. Note that for degree reasons alone,  $n=m$. Also if $n=1$, then there is nothing to prove. On the other hand, for $n>1$ let $F=\theta_2\circ\dots\circ \theta_n$, let $G=\tau_2\circ\dots\circ \tau_n$, let $\theta_1=x^d+b_1$, and let $\tau_1=x^d+b_2$. Then, $F^d-G^d=b_2-b_1$, and so if $b_2\neq b_1$, we obtain non-constant solutions $(X,Y)=(F,G)$ to the Fermat-Catalan equation $X^d-Y^d=(b_2-b_1)$. However, Mason's $abc$-theorem implies that there are no such solutions; see, for instance, \cite[Lemma 3.2]{hindes2021orbit}. 
Therefore, we deduce that $b_1=b_2$ and $F^d=G^d$. But then $F/G$ is a $d$th root of unity in $K$ so that we may write $F=\zeta G$ for some constant $\zeta\in K$. However, $F$ and $G$ are both monic, which implies that $\zeta=1$. In summary: we have shown that $n=m$, $\theta_1=\tau_1$, and $\theta_2\circ\dots\circ \theta_n=\tau_2\circ\dots\circ \tau_n$. In particular, we may continue on inductively in this way to deduce that $\theta_i=\tau_i$ for all $i\geq1$ as desired.           
\end{proof}
With the previous result in mind, we may define the length of an element $F\in \mathcal{U}_d$.  
\begin{defin} Let $K$ be a field of characteristic zero and let $\mathcal{U}_d$ be as in Proposition \ref{prop:free}. Then we define the \emph{length} of $F\in \mathcal{U}_d$ to be $n$ if $F=\theta_1\circ\dots\circ\theta_n\in G$, where each $\theta_i$ is of the form $x^d+c_i$ for some $c_i\in K$. In this case, we write $\ell(F)=n$.     
\end{defin}
We now have all of the tools in place to prove our main irreducibility result. Namely, outside of a small one-parameter family of exceptions, a semigroup $G=\langle x^d+c_1,\dots,x^d+c_s\rangle$ for some $c_1,\dots,c_s\in\mathbb{Z}$ and some $d\geq2$ contains a positive proportion of irreducible polynomials if and only if it contains at least one such polynomial. 
\begin{thm}\label{thm:main+irreducible+more+specific} Let $G=\langle x^d+c_1,\dots, x^d+c_s\rangle$ for some $d \ge 2$ and $c_1,\dots,c_s\in\mathbb{Z}$ and assume that $G$ contains an irreducible polynomial. Then one of the following statements must hold: \vspace{.1cm} 
\begin{enumerate}
    \item[\textup{(1)}] There exists $F\in G$ with $\ell(F)\leq 6$ such that $\{F\circ g\,:\, g\in G\}$ is a set of irreducible polynomials. \vspace{.3cm}     
    \item[\textup{(2)}] $d\geq4$ is even and $\{c_1,\dots, c_s\}\subseteq \big\{y^p-y^{pd}\,,\;y^p\,,\;-y^p\,,\;-y^p-y^{pd}\big\}$ for some $y\in\mathbb{Z}$ and some prime $p|d$. \vspace{.3cm} 
      \item[\textup{(3)}] $d\geq5$ is odd and $\{c_1,\dots,c_s\}=\big\{y^p-y^{pd}\;,\;y^p\big\}$ for some $y\in\mathbb{Z}$ and some prime $p|d$.   
\end{enumerate}
\end{thm} 
\begin{proof}
Let $G=\langle x^d+c_1,\dots, x^d+c_s\rangle$ for some $d \ge 2$ and $c_1,\dots,c_s\in\mathbb{Z}$. The case when $d=2$ follows from \cite[Theorem 1.4]{hindes2025proportion}. In fact, in this case one may find $F\in G$ with $\ell(F)\leq5$ with the desired property. Now suppose that $d\geq3$ and that $G$ contains an irreducible polynomial. Then one of the generators, without loss say $f_1=x^d+c_1$, must also be irreducible. Note also that we may assume that $G$ has at least two generators, since otherwise $G=\langle f_1\rangle$ consists entirely of irreducible polynomials by Proposition \ref{prop:stability}. Now if $f_1$ has no $p$th powered fixed point for some prime $p|d$, then $F=f_1^3$ has the desired property in statement (1) by Proposition \ref{prop:no+special+irreducibles}. Therefore, we may assume that $f_1$ has a fixed point $y^p$ for some $y\in\mathbb{Z}$ and some prime $p|d$. Thus $c_1=y^p-y^{pd}$.

From here, suppose that $d$ is odd. If $\{c_1,\dots,c_s\}\not\subseteq \big\{y^p-y^{pd}\;,\;y^p\big\}$, then there exists $f_2=x^d+c_2\in  G$ such that either $F=f_1^3\circ f_2^3$ or $F=f_1^3\circ f_2\circ f_1$ has the desired property in statement (1) by Proposition \ref{prop:special+irreducible+odd}. Hence, we may assume that $\{c_1,\dots,c_s\}=\big\{y^p-y^{pd}\;,\;y^p\big\}$. Moreover, when $d=3$ the polynomial $F=f_1\circ f_2\circ f_1$ satisfies statement (1) by \cite[Corollary 4.10]{hindes2025proportion}. Therefore, if no such $F$ exists when $d$ is odd, then $\{c_1,\dots,c_s\}=\big\{y^p-y^{pd}\;,\;y^p\big\}$ and $d\geq5$ as in statement (3). 

From here, suppose that $d\geq 4$ is even. If $\{c_1,\dots,c_s\}\not\subseteq \big\{y^p-y^{pd}\,,\;y^p\,,\;-y^p\,,\;-y^p-y^{pd}\big\}$, then there exists $f_2=x^d+c_2\in  G$ such that either $F=f_1^3\circ f_2^3$ or $F=f_1^3\circ f_2\circ f_1$ has the desired property in statement (1) by Proposition \ref{prop:special+irreducible+odd}. Hence, we may assume that $\{c_1,\dots, c_s\}\subseteq \big\{y^p-y^{pd}\,,\;y^p\,,\;-y^p\,,\;-y^p-y^{pd}\big\}$ as in statement (2).  
\end{proof}
\begin{remark} We note that Theorem \ref{thm:main+irreducible+more+specific} implies Theorem \ref{thm:main+irreducible} from the Introduction. Indeed, if $G$ contains an irreducible polynomial and is not generated by polynomials of the form in statement (2) or statement (3) of Theorem \ref{thm:main+irreducible}, then Theorem \ref{thm:main+irreducible+more+specific} implies that there exists $F\in G$ with $\ell(F)\leq6$ such that $\{F\circ g\,:\, g\in G\}$ is a set of irreducible polynomials. But then it is straightforward to check that,
\[\liminf_{B\rightarrow\infty}\frac{\#\{g\in G\,: \deg(g)\leq B\;\text{and $g$ is irreducible over $K$}\}}{\#\{g\in G\,: \deg(g)\leq B\}}\geq \frac{1}{s^6}>0,\]
where $s$ is the number of generators of $G$; here we also use that $G$ is free by Proposition \ref{prop:free}.             
\end{remark}
\bibliographystyle{plain}
\bibliography{Main}
\end{document}